\documentclass{amsart}
\usepackage{amssymb}
\usepackage{amsbsy, amsthm, amsmath,amstext, amsopn}
\usepackage[all]{xy}
\usepackage{amsfonts}
\usepackage{amscd}
\usepackage{stmaryrd}
\newtheorem{thm}{Theorem} [section]
\newtheorem{lemma}[thm]{Lemma}

\newtheorem{corollary}[thm]{Corollary}
\newtheorem{prop}[thm]{Proposition}

\theoremstyle{definition}

\newtheorem{defn}[thm]{Definition}

\theoremstyle{remark}

\newtheorem{remark}[thm]{Remark}

\begin{document}

\numberwithin{equation}{section}

\newcommand{\hs}{\mbox{\hspace{.4em}}}
\newcommand{\ds}{\displaystyle}
\newcommand{\bd}{\begin{displaymath}}
\newcommand{\ed}{\end{displaymath}}
\newcommand{\bcd}{\begin{CD}}
\newcommand{\ecd}{\end{CD}}

\newcommand{\on}{\operatorname}
\newcommand{\proj}{\operatorname{Proj}}
\newcommand{\bproj}{\underline{\operatorname{Proj}}}

\newcommand{\spec}{\operatorname{Spec}}
\newcommand{\Spec}{\operatorname{Spec}}
\newcommand{\bspec}{\underline{\operatorname{Spec}}}
\newcommand{\pline}{{\mathbf P} ^1}
\newcommand{\aline}{{\mathbf A} ^1}
\newcommand{\pplane}{{\mathbf P}^2}
\newcommand{\aplane}{{\mathbf A}^2}
\newcommand{\coker}{{\operatorname{coker}}}
\newcommand{\ldb}{[[}
\newcommand{\rdb}{]]}

\newcommand{\Sym}{\operatorname{Sym}^{\bullet}}
\newcommand{\Symp}{\operatorname{Sym}}
\newcommand{\Pic}{\bf{Pic}}
\newcommand{\Aut}{\operatorname{Aut}}
\newcommand{\PAut}{\operatorname{PAut}}

\newcommand{\too}{\twoheadrightarrow}
\newcommand{\C}{{\mathbf C}}
\newcommand{\Z}{{\mathbf Z}}
\newcommand{\Q}{{\mathbf Q}}
\newcommand{\R}{{\mathbf R}}
\newcommand{\Cx}{{\mathbf C}^{\times}}
\newcommand{\Cbar}{\overline{\C}}
\newcommand{\Cxbar}{\overline{\Cx}}
\newcommand{\cA}{{\mathcal A}}
\newcommand{\cS}{{\mathcal S}}
\newcommand{\cV}{{\mathcal V}}
\newcommand{\cM}{{\mathcal M}}
\newcommand{\bA}{{\mathbf A}}
\newcommand{\cB}{{\mathcal B}}
\newcommand{\cC}{{\mathcal C}}
\newcommand{\cD}{{\mathcal D}}
\newcommand{\D}{{\mathcal D}}
\newcommand{\cs}{{\mathbf C} ^*}
\newcommand{\boldc}{{\mathbf C}}
\newcommand{\cE}{{\mathcal E}}
\newcommand{\cF}{{\mathcal F}}
\newcommand{\bF}{{\mathbf F}}
\newcommand{\cG}{{\mathcal G}}
\newcommand{\G}{{\mathbb G}}
\newcommand{\cH}{{\mathcal H}}
\newcommand{\CI}{{\mathcal I}}
\newcommand{\cJ}{{\mathcal J}}
\newcommand{\cK}{{\mathcal K}}
\newcommand{\cL}{{\mathcal L}}
\newcommand{\baL}{{\overline{\mathcal L}}}
\newcommand{\M}{{\mathcal M}}
\newcommand{\Mf}{{\mathfrak M}}
\newcommand{\bM}{{\mathbf M}}
\newcommand{\bm}{{\mathbf m}}
\newcommand{\cN}{{\mathcal N}}
\newcommand{\theo}{\mathcal{O}}
\newcommand{\cP}{{\mathcal P}}
\newcommand{\cR}{{\mathcal R}}
\newcommand{\Pp}{{\mathbb P}}
\newcommand{\boldp}{{\mathbf P}}
\newcommand{\boldq}{{\mathbf Q}}
\newcommand{\bbL}{{\mathbf L}}
\newcommand{\cQ}{{\mathcal Q}}
\newcommand{\cO}{{\mathcal O}}
\newcommand{\Oo}{{\mathcal O}}
\newcommand{\cY}{{\mathcal Y}}
\newcommand{\OX}{{\Oo_X}}
\newcommand{\OY}{{\Oo_Y}}
\newcommand{\otY}{{\underset{\OY}{\ot}}}
\newcommand{\otX}{{\underset{\OX}{\ot}}}
\newcommand{\cU}{{\mathcal U}}\newcommand{\cX}{{\mathcal X}}
\newcommand{\cW}{{\mathcal W}}
\newcommand{\boldz}{{\mathbf Z}}
\newcommand{\qgr}{\operatorname{q-gr}}
\newcommand{\gr}{\operatorname{gr}}
\newcommand{\rk}{\operatorname{rk}}
\newcommand{\Sh}{\operatorname{Sh}}
\newcommand{\SH}{{\underline{\operatorname{Sh}}}}
\newcommand{\End}{\operatorname{End}}
\newcommand{\uEnd}{\underline{\operatorname{End}}}
\newcommand{\Hom}{\operatorname{Hom}}
\newcommand{\uHom}{\underline{\operatorname{Hom}}}
\newcommand{\uHomY}{\uHom_{\OY}}
\newcommand{\uHomX}{\uHom_{\OX}}
\newcommand{\Ext}{\operatorname{Ext}}
\newcommand{\bExt}{\operatorname{\bf{Ext}}}
\newcommand{\Tor}{\operatorname{Tor}}

\newcommand{\inv}{^{-1}}
\newcommand{\airtilde}{\widetilde{\hspace{.5em}}}
\newcommand{\airhat}{\widehat{\hspace{.5em}}}
\newcommand{\nt}{^{\circ}}
\newcommand{\del}{\partial}

\newcommand{\supp}{\operatorname{supp}}
\newcommand{\GK}{\operatorname{GK-dim}}
\newcommand{\hd}{\operatorname{hd}}
\newcommand{\id}{\operatorname{id}}
\newcommand{\res}{\operatorname{res}}
\newcommand{\lrar}{\leadsto}
\newcommand{\im}{\operatorname{Im}}
\newcommand{\HH}{\operatorname{H}}
\newcommand{\TF}{\operatorname{TF}}
\newcommand{\Bun}{\operatorname{Bun}}

\newcommand{\F}{\mathcal{F}}
\newcommand{\Ff}{\mathbb{F}}
\newcommand{\nthord}{^{(n)}}
\newcommand{\Gr}{{\mathfrak{Gr}}}

\newcommand{\Fr}{\operatorname{Fr}}
\newcommand{\GL}{\operatorname{GL}}
\newcommand{\gl}{\mathfrak{gl}}
\newcommand{\SL}{\operatorname{SL}}
\newcommand{\ff}{\footnote}
\newcommand{\ot}{\otimes}
\def\Ext{\operatorname {Ext}}
\def\Hom{\operatorname {Hom}}
\def\Ind{\operatorname {Ind}}
\def\bbZ{{\mathbb Z}}

\newcommand{\nc}{\newcommand}
\nc{\ol}{\overline} \nc{\cont}{\on{cont}} \nc{\rmod}{\on{mod}}
\nc{\Mtil}{\widetilde{M}} \nc{\wb}{\overline} \nc{\wt}{\widetilde}
\nc{\wh}{\widehat} \nc{\sm}{\setminus} \nc{\mc}{\mathcal}
\nc{\mbb}{\mathbb}  \nc{\K}{{\mc K}} \nc{\Kx}{{\mc K}^{\times}}
\nc{\Ox}{{\mc O}^{\times}} \nc{\unit}{{\bf \on{unit}}}
\nc{\boxt}{\boxtimes} \nc{\xarr}{\stackrel{\rightarrow}{x}}

\nc{\Ga}{\G_a}
 \nc{\PGL}{{\on{PGL}}}
 \nc{\PU}{{\on{PU}}}

\nc{\h}{{\mathfrak h}} \nc{\kk}{{\mathfrak k}}
 \nc{\Gm}{{\G_m}}
\nc{\Gabar}{\wb{\G}_a} \nc{\Gmbar}{\wb{\G}_m} \nc{\Gv}{G^\vee}
\nc{\Tv}{T^\vee} \nc{\Bv}{B^\vee} \nc{\g}{{\mathfrak g}}
\nc{\gv}{{\mathfrak g}^\vee} \nc{\RGv}{\on{Rep}\Gv}
\nc{\RTv}{\on{Rep}T^\vee}
 \nc{\Flv}{{\mathcal B}^\vee}
 \nc{\TFlv}{T^*\Flv}
 \nc{\Fl}{{\mathfrak Fl}}
\nc{\RR}{{\mathcal R}} \nc{\Nv}{{\mathcal{N}}^\vee}
\nc{\St}{{\mathcal St}} \nc{\ST}{{\underline{\mathcal St}}}
\nc{\Hec}{{\bf{\mathcal H}}} \nc{\Hecblock}{{\bf{\mathcal
H_{\alpha,\beta}}}} \nc{\dualHec}{{\bf{\mathcal H^\vee}}}
\nc{\dualHecblock}{{\bf{\mathcal H^\vee_{\alpha,\beta}}}}
\newcommand{\ramBun}{{\bf{Bun}}}
\newcommand{\ramBuno}{\ramBun^{\circ}}

\nc{\Buntheta}{{\bf Bun}_{\theta}} \nc{\Bunthetao}{{\bf
Bun}_{\theta}^{\circ}} \nc{\BunGR}{{\bf Bun}_{G_\R}}
\nc{\BunGRo}{{\bf Bun}_{G_\R}^{\circ}}
\nc{\HC}{{\mathcal{HC}}}
\nc{\risom}{\stackrel{\sim}{\to}} \nc{\Hv}{{H^\vee}}
\nc{\bS}{{\mathbf S}}
\def\Rep{\operatorname {Rep}}
\def\Conn{\operatorname {Conn}}

\nc{\Vect}{{\operatorname{Vect}}}
\nc{\Hecke}{{\operatorname{Hecke}}}

\newcommand{\ZZ}{{Z_{\bullet}}}
\nc{\HZ}{{\mc H}\ZZ} \nc{\eps}{\epsilon}

\nc{\CN}{\mathcal N} \nc{\BA}{\mathbb A}

\nc{\ul}{\underline}

\nc{\bn}{\mathbf n} \nc{\Sets}{{\on{Sets}}} \nc{\Top}{{\on{Top}}}
\nc{\IntHom}{{\mathcal Hom}}

\nc{\Simp}{{\mathbf \Delta}} \nc{\Simpop}{{\mathbf\Delta^\circ}}

\nc{\Cyc}{{\mathbf \Lambda}} \nc{\Cycop}{{\mathbf\Lambda^\circ}}

\nc{\Mon}{{\mathbf \Lambda^{mon}}}
\nc{\Monop}{{(\mathbf\Lambda^{mon})\circ}}

\nc{\Aff}{{\on{Aff}}} \nc{\Sch}{{\on{Sch}}}

\nc{\bul}{\bullet}
\nc{\module}{{\operatorname{mod}}}

\nc{\dstack}{{\mathcal D}}

\nc{\BL}{{\mathbb L}}

\nc{\BD}{{\mathbb D}}

\nc{\BR}{{\mathbb R}}

\nc{\BT}{{\mathbb T}}

\nc{\SCA}{{\mc{SCA}}}
\nc{\DGA}{{\mc DGA}}

\nc{\DSt}{{DSt}}

\nc{\lotimes}{{\otimes}^{\mathbf L}}

\nc{\bs}{\backslash}

\nc{\Lhat}{\widehat{\mc L}}

\newcommand{\Coh}{{\on{Coh}}}

\nc{\QCoh}{QC}
\nc{\QC}{QC}
\nc{\Perf}{\rm{Perf}}
\nc{\Cat}{{\on{Cat}}}
\nc{\dgCat}{{\on{dgCat}}}
\nc{\bLa}{{\mathbf \Lambda}}

\nc{\RHom}{\mathbf{R}\hspace{-0.15em}\on{Hom}}
\nc{\REnd}{\mathbf{R}\hspace{-0.15em}\on{End}}
\nc{\colim}{\on{colim}}
\nc{\oo}{\infty}
\nc{\Mod}{{\on{Mod}}}

\nc\fh{\mathfrak h}
\nc\al{\alpha}
\nc\la{\alpha}
\nc\BGB{B\bs G/B}
\nc\QCb{QC^\flat}
\nc\qc{QC}

\nc{\fg}{\mathfrak g}

\nc{\Map}{\on{Map}} \nc{\fX}{\mathfrak X}

\nc{\ch}{\check}
\nc{\fb}{\mathfrak b} \nc{\fu}{\mathfrak u} \nc{\st}{{st}}
\nc{\fU}{\mathfrak U}
\nc{\fZ}{\mathfrak Z}

\nc\fk{\mathfrak k} \nc\fp{\mathfrak p}

\nc{\RP}{\mathbf{RP}} \nc{\rigid}{\text{rigid}}
\nc{\glob}{\text{glob}}

\nc{\cI}{\mathcal I}

\nc{\La}{\mathcal La}

\nc{\quot}{/\hspace{-.25em}/}

\nc\aff{\it{aff}}
\nc\BS{\mathbb S}

\nc{\qcoh}{\on{QCoh}}
\nc{\topmodule}{\on{-cmod}}
\nc{\fingen}{\on{fin}}
\nc{\fgmodule}{\on{fgmod}}

\nc\coh{\text{\it \hspace{-0.5em} coh}}

\title{Loop Spaces and Representations}

\author{David Ben-Zvi}
\address{Department of Mathematics\\University of Texas\\Austin, TX 78712-0257}
\email{benzvi@math.utexas.edu}
\author{David Nadler}
\address{Department of Mathematics\\Northwestern University\\Evanston, IL 60208-2370}
\email{nadler@math.northwestern.edu}

\begin{abstract}
  We introduce loop spaces (in the sense of derived algebraic
  geometry) into the representation theory of reductive groups.  In
  particular, we apply the theory developed in~\cite{conns} to flag
  varieties, and obtain new insights into fundamental categories in
  representation theory. First, we show that one can recover
  finite Hecke categories (realized by $\D$-modules on flag varieties)
  from affine Hecke categories (realized by coherent sheaves on
  Steinberg varieties) via $S^1$-equivariant localization. Similarly,
  one can recover $\D$-modules on the nilpotent cone from coherent sheaves
  on the commuting variety. We also show that the categorical
  Langlands parameters for real groups studied by
  Adams-Barbasch-Vogan~\cite{ABV} and Soergel~\cite{Soergel} arise
  naturally from the study of loop spaces of flag varieties and their
  Jordan decomposition (or in an alternative formulation, from the
  study of local systems on a M\"obius strip). This provides a
  unifying framework that overcomes a discomforting aspect of the
  traditional approach to the Langlands parameters, namely their
  evidently strange behavior with respect to changes in infinitesimal
  character.
\end{abstract}

\maketitle

\tableofcontents


\section{Introduction}

This is the second paper in a two paper series\footnote{It is a
  strengthened version of the second half of the
  preprint~\cite{BN07}.}. In the first paper~\cite{conns}, we studied
loop spaces in derived algebraic geometry, in particular the relation
between quasicoherent sheaves on loop spaces and connections on the
original space.  In this paper, we apply this theory to flag varieties
of reductive groups and find consequences for the representation
theory of Hecke algebras and Lie groups.

It is well known that connections -- in the form of $\D$-modules --
play a central role in geometric representation theory. The theory of
the first part~\cite{conns} establishes the intimate relation between
$\D$-modules and the geometry of so called ``small loops,'' or more
precisely, loops in the formal neighborhood of constant loops.  Thus
with Beilinson-Bernstein localization in mind, it is not surprising
that small loops provide an alternative language to discuss
constructions of representations. The striking feature of this second
part is that the geometry of the space of all loops arises naturally
in the dual {\em Langlands parametrization} of
representations. Essential to this realization is the notion,
motivated by rational homotopy theory, of {\em unipotent loops}
introduced in \cite{conns}. In fact, we describe a pattern of Jordan
decomposition for loop spaces which precisely accounts for the
seemingly erratic behavior of the Langlands parameters as functions of
infinitesimal character. Thus we will see that the notion of loop
space organizes much of the seeming cacophany of spaces parametrizing
representations of Lie groups.

We summarize our main results immediately below.  We have organized
the dicsussion into three parts: the first related to Hecke categories,
the second to representations of real groups, and the third to the
nilpotent cone and commuting variety.

In Section~\ref{section: background}, we briefly review the results of
\cite{conns} on loop spaces in derived algebraic geometry. This will
provide the bridge from quasicoherent sheaves to $\D$-modules.

In Section~\ref{Steinberg section}, we present our results on
Steinberg varieties and Hecke categories.

In Section~\ref{Langlands section}, we present our results on
Langlands parameters for real groups.

\medskip
\noindent
{\em Categorical conventions.}  We work throughout over the complex
numbers $\C$.  Our main results are equivalences of small
pre-triangulated $\C$-linear differential graded (dg)
categories. (Typically these arise as dg enhancements of derived
categories of modules for differential graded algebras.)
We will consistently view such dg categories as objects of the
$\oo$-category of small idempotent-complete stable $\C$-linear
$\oo$-categories (with morphisms exact functors), or (passing to
ind-categories) of stable presentable $\C$-linear $\oo$-categories
(with morphisms continuous functors).  Where not made explicit, the term {\em dg category} will refer to such an enhanced derived category, so localized along quasi-isomorphisms, and all functors will be assumed to be derived. 
For example, by the dg category  of quasicoherent sheaves on a stack, we will mean the standard dg enhancement of the derived category. Below in Section~\ref{section: background}, one will find a further summary of 
adopted conventions from derived
algebraic geometry.

\medskip
\noindent
{\em Lie notation.}
For easy reference, we collect here a brief summary of our Lie notation.
Let $ G$ be a connected reductive
complex algebraic group with Lie algebra $ \fg$. 
Let $ \cB$ be
the flag variety of $ G$ parameterizing Borel subgroups $ B\subset
G$. For each $ B\in  \cB$, we have an identification $\cB \simeq G/B$, and the Cartan quotient $ H= B/ U$
where $ U\subset  B$ is the unipotent radical. The natural conjugation $
G$-action on $\cB$ canonically identifies the Cartan
quotients for different $ B$, and so it is justified to call $ H$ the universal
Cartan.
It is also convenient to choose a maximal torus $T\subset B$ with normalizer $N(T)\subset G$,
and Weyl group $W= N(T)/T$.


\medskip
\noindent
{\em Coherent sheaves.}  We will frequently encounter closed
embeddings of stacks $e:X \hookrightarrow Z$ in which $Z$ is equipped
with a contracting $\Gm$-action with fixed locus $X$. By contracting, we mean the action map
$\Gm\times X \to X$ extends to a map $\BA^1 \times X\to X$ such that $\{0\} \times Z \to Z$ is a retraction onto $X$.
 Furthermore, the stacks will be
quotients by linear actions of an affine group scheme: $X$ will be the
quotient of a smooth quasi-projective variety; $Z$ will be the
quotient of the formal scheme given by the formal neighborhood of an
invariant subvariety of a smooth quasi-projective variety.  We will
write $\Coh_{[X]}(Z)^\Gm$ for the dg derived category of
$\Gm$-equivariant coherent sheaves $\cF$ on $Z$ such that the
restriction $e^*\cF$ is also coherent (or equivalently, perfect since
$X$ is smooth).  We refer to $\Coh_{[X]}(Z)^\Gm$ as the dg category of
$\Gm$-equivariant coherent sheaves on $Z$ coherent along $X$.  To see
that $\Coh_{[X]}(Z)^\Gm$ is a rich category, one could note that any
$\Gm$-equivariant perfect sheaf on $Z$ is an object of
$\Coh_{[X]}(Z)^\Gm$.

%


\medskip
\noindent{\em Acknowledgments.} We would like to
thank Roman Bezrukavnikov for his continued interest, many helpful
discussions and for sharing with us unpublished work.

The first author was partially supported by NSF grant DMS-0449830,
and the second author by NSF grant DMS-0600909. Both authors would
also like to acknowledge the support of DARPA grant
HR0011-04-1-0031.

\subsection{Hecke categories}

Our first main result provides a direct relation between two of the
fundamental categories in representation theory.
  
\begin{defn}
  
  The {\em finite Hecke category} is the
  dg category $$\cH_G=\cD_{\coh}(\BGB)
  $$ of $B$-equivariant coherent $\D$-modules on the
  flag variety $\cB \simeq G/B$.
  \end{defn}
  
   Let $G^u\subset G$ denote the formal scheme given by the formal
   neighborhood of the unipotent elements of $G$, and recall the
   unipotent Steinberg variety (or more precisely, formal scheme, though the 
   technical distinction will play no role)
$$\St^u=\{(g, B_1, B_2) \in G^u\times \cB\times\cB\, |\, g\in B_1\cap B_2\}$$
Requiring the element $g\in G^u$ to be the identity provides a canonical closed embedding 
$$ \xymatrix{ \cB \times \cB \ar@{^(->}[r] & \St^u }$$ 
Equivalently, instead of $G^u \subset G$, we could take the Lie algebra version $\fg^u \subset \fg$. In this way,
we see that there
is a natural contracting $\Gm$-action on $\St^u$ induced by the dilation of $\fg$. 
  \begin{defn}
  
  We take the {\em affine Hecke category} to be
 the dg
  derived category  
  $$
  \cH_G^{\aff}=\Coh_{[(\cB\times  \cB)/G]}(\St^u/G)^\Gm
  $$ 
  of $\Gm$-equivariant coherent sheaves 
  on  $ \St^u/G$ coherent along $(\cB\times \cB)/G$.

\end{defn}

\begin{remark}
One might prefer to take the 
affine Hecke category to be
 the dg
  derived category  
  of all $\Gm$-equivariant coherent sheaves on $\St^u/G$.
  We have only restricted to the full subcategory of the above definition in order to state our main results most cleanly.
\end{remark}

The starting point of this work is the following observation: one can
realize $\St^u/G$ as a kind of loop space in the setting of derived
algebraic geometry. Hence it carries a corresponding action of the
circle $S^1$, which can be expressed  in terms of the
Hopf algebra of cochains 
$$
\xymatrix{
\cO(S^1) = C^*(S^1,k)\simeq H^*(S^1,k)=k[\eta]/(\eta^2),\quad |\eta|=1.
}
$$
More precisely, $\St^u/G$ is a {\em unipotent loop space}, a notion
introduced in \cite{conns} (with motivation from rational homotopy
theory) and reviewed in Section \ref{section: background}.  As a
result, the circle action factors through an action of the affinization
$$\Aff(S^1) = \Spec \cO(S^1).
$$
Moreover, the dilation $\Gm$-action on $\St^u/G$ is naturally induced
by the formality of $\cO(S^1)$.  It follows (as for any unipotent loop
space) that the $S^1$-action is compatible with the $\Gm$-action, and
they combine into the action of the single semi-direct product group
$$
\BS = \Aff(S^1) \rtimes \Gm.
$$

Recall that if a category $\cD$ has an $S^1$-action, the category
$\cC=\cD^{S^1}$  of $S^1$-equivariant objects is
naturally linear over the global functions
$$
\cO(BS^1) \simeq H^*(BS^1, k) \simeq k[u],\quad |u|=2
$$ 
We write
$\cC_{loc}$ for the localization of $\cC$ where we invert the action
of the generator $u$.  Since the cohomological degree of $u$ is two,
the morphisms of $\cC_{loc}$ are naturally $\Z/2\Z$-graded rather than
$\Z$-graded.  If $\cD$ has an $\BS$-action, then
 the category $\cC = \cD^\BS$ of $\BS$-equivariant objects  is also
 naturally linear over $\cO(BS^1)$. But here
  thanks to the additional $\Gm$-weight grading, 
 the morphisms of $\cC_{loc}$  are
 again naturally $\Z$-graded.

Now we arrive at our first main result. 

\begin{thm} \label{first main result}There is a canonical equivalence $$\cD_{\coh}(\BGB)\simeq
 \Coh_{[(\cB\times  \cB)/G]}(\St^u/G)^\BS_{loc}
 $$ between the finite Hecke
  category and the localized $\BS$-invariants of the
  affine Hecke category. 
\end{thm}

\begin{remark}
Although we do not discuss monoidal structures in this paper, one can
check that the above equivalence is naturally monoidal with respect to
convolution.
\end{remark}

The proof of Theorem~\ref{first main result} involves two parts:
first, making precise the relation between $\St^u/G$ and the loop
space of $B\bs G/B$, and second, applying general results on loop
spaces and connections from \cite{conns}.

A striking aspect of the theorem is its mixture of quasicoherent sheaves
and $\D$-modules.  We were led to it via the duality of local tamely
ramified Hecke operators in the Geometric Langlands program. In what
immediately follows, we informally explain this motivating picture.

Let $G^\vee$ denote the Langlands dual group, let $LG^\vee$ be its
loop group of maps $\Spec \C((t)) \to G^\vee$, and let $I \subset
LG^\vee$ be an Iwahori subgroup. The traditional definition of the
affine Hecke category is the dg category $D_u(I \bs
LG^\vee/I)$ of perverse sheaves on $LG^\vee$ which are
Iwahori-bimonodromic with pro-unipotent monodromy.  A deep theorem of
Bezrukavnikov~\cite{Roma ICM} (with groundbreaking applications to
Lusztig's vision of the representation theory of Lie algebras in
characteristic $p$, quantum groups, and affine algebras) lifts the
Kazhdan-Lustig realizations of affine Hecke algebras \cite{KL} (or see
\cite{CG} for an exposition) to the level of derived categories
$$
\Coh(\St^u/G) \simeq D_u(I \bs LG^\vee/I).
$$

Our discovery of Theorem~\ref{first main result} was inspired by a
more evident parallel picture on the other side of the above
equivalence.  Namely, by its very definition $LG^\vee$ is a loop
space, and loop rotation equips $D_u(I \bs LG^\vee/I)$ with a natural
locally constant $\Gm$-action, which factors through an
$S^1$-action. The rotation fixed points in $LG^\vee$ are the constant
loops $G^\vee$, and one can show that the localized $S^1$-invariants
in $D_u(I \bs LG^\vee/I)$ form (a periodic version of) the
$B$-bimonodromic finite Hecke category $D_u(B^\vee \bs
G^\vee/B^\vee)$. This is a relatively straightforward application of
equivariant localization in topology, for example using the
differential graded techniques developed by
Goresky-Kottwitz-MacPherson~\cite{GKM}.

Soergel~\cite{Soergel} interpreted the Koszul duality theorem of
Beilinson-Ginzburg-Soergel~ \cite{BGS} in the context of Langlands
duality, identifying (graded or 2-periodic versions of) the finite
Hecke categories $D_u(B^\vee \bs G^\vee/B^\vee)$ and $\cD(\BGB)$.  It
is interesting to note that as a consequence of the above perspective,
we recover the Langlands duality of finite Hecke categories from
Bezrukavnikov's equivalence of affine Hecke categories by equivariant
localization. In particular, this fixes a canonical monoidal form of
this duality. (For a direct verification of monoidal aspects of this
equivalence, see Bezrukavnikov-Yun~\cite{BY}.)


\subsection{Langlands parameters}

Our second main result gives a uniform geometric description of the
categorical Langlands parameters for representations of real groups.

Let $G^\vee$ denote the Langlands dual group, and let $\theta$ be a
quasi-split conjugation of $G^\vee$.  Adams-Barbasch-Vogan \cite{ABV}
and Soergel \cite{Soergel} introduce and investigate categories of
equivariant $\D$-modules which parametrize Harish Chandra modules for
real forms of $G^\vee$ in the inner class of $\theta$.  For a fixed
infinitesimal character, the $\D$-modules live on flag varieties
constructed out of $G$, an involution $\eta$ corresponding to
$\theta$, and the infinitesimal character.  On the level of
Grothendieck groups, the parametrization is a form of Vogan's
character duality ~\cite{Vogan}, while on the level of derived
categories, it is a Koszul duality conjecture due to
Soergel~\cite{Soergel}.

We have aimed to understand this circle of ideas through the lens of
topological field theory.  For example, in the paper \cite{character},
we construct a partial 3d TFT which governs the above objects and has
a natural Langlands duality. Our aim here is to deduce that the
parameter categories of \cite{ABV} and \cite{Soergel} naturally arise
from the 4d Geometric Langlands TFT.  We will not go into the details
of TFT, but rather highlight some of the consequences of this
viewpoint, in particular a new connection of the parameter categories
to affine Hecke categories and a new uniform description of them with
respect to regular infinitesimal character.

\begin{defn}
For an involution $\eta$ of $G$, the {\em Langlands parameter variety}
$\La^\eta$ is defined to be the formal scheme
$$
\La^\eta = \{ (g, B) \in G\times \cB | g\eta(g) \in B\}.
$$ 

For an element $\alpha$ of the universal Cartan $H$, we define the
{\em monodromic Langlands parameter variety} $\La^\eta_\alpha \subset \cL$
to be the subscheme of pairs $(g, B)$ such that the semisimple part of
$g\eta(g)$, with respect to the flag $B$, is in the formal
neighborhood of $\alpha$.
\end{defn}

The group $G$ acts on $\cL$ by twisted conjugation, and one can realize
$\La^\eta/G$ as a kind of loop space. Thanks to a Jordan
decomposition pattern for loops, each of the substacks $\La^\eta_\alpha/G$ has a
natural loop space interpretation resulting in an $\BS$-action on
$\cL_\alpha/G$ with resulting $\Gm$-action contracting to a fixed substack $X_\alpha/G$. (We will make these structures more  precise after developing further notation below.)

\begin{remark}
One can realize $\cL/G$  as
a moduli  of local systems on the M\"obius strip with a flag
along the boundary. From this viewpoint, the $S^1$-action on $\cL/G$ comes from rotating the M\"obius strip.
It is this picture, originating in
topological field theory, which motivated the definition of $\cL/G$.

In general, there
  is no evident compatible $\Gm$-action. We will explain that in fact
  $\La^\eta_\alpha/G$ is itself a unipotent loop space and thus has a canonical
  $\BS$-action. The corresponding $S^1$-action coincides up to a
  central automorphism with that obtained by rotating the M\"obius
  strip.
\end{remark}

We will consider the dg category $\Coh_{[X_\alpha/G]}(\La^\eta_\alpha/G)^\Gm$ of
$\Gm$-equivariant coherent sheaves on the
stack $\La^\eta_\alpha/G$ coherent along $X_\alpha/G$. 

\begin{remark}
Though we do not investigate this structure here, it is
noteworthy that the dg category $\Coh_{[X_\alpha/G]}(\La^\eta_\alpha/G)^\Gm$ is naturally a module
for the appropriate monodromic affine Hecke category.
\end{remark}

Our second main result is the following theorem. We state it more
precisely in Theorem~\ref{categorical Langlands precise} immediately
below after developing further notation.

\begin{thm}[Informal version]\label{cite categorical Langlands}
For any regular $\lambda\in \fh$, the Langlands parameter category for
infinitesimal character $[\lambda] \in \fh/W$ is canonically equivalent to the localized $\BS$-invariants 
$\Coh_{[X_\alpha/G]}(\La^\eta_\alpha/G)^\BS_{loc}$ where $\alpha = \exp(\lambda)\in H$.
\end{thm}

\begin{remark}
In the case of a complex group (considered as a real form
of its complexification), the theorem is a version (with parameter
$\alpha\in H\times H$) of Theorem \ref{first main result} (to which it
reduces when $\alpha=(e,e)$) which is spelled out in detail in Section
\ref{application to Hecke}.
\end{remark}

The precise version of Theorem~\ref{cite categorical Langlands} proved in this
paper (see Corollaries~\ref{trivmonresult} and ~\ref{genmonresult}) is self-contained and makes no reference to the theory of
\cite{ABV} and \cite{Soergel}. It gives a direct and concrete
description of $\Coh_{[X_\alpha/G]}(\La^\eta_\alpha/G)^\BS_{loc}$ in terms of
equivariant $\D$-modules on flag varieties.  One can then check that
our description coincides with the corresponding Langlands parameter
category of \cite{ABV} and \cite{Soergel} for regular infinitesimal
character.

 Now in order to state a more precise version of Theorem~\ref{cite
   categorical Langlands}, we include here a slightly expanded review
 of the Langlands parameter categories for representations of real
 groups.  Recall that $G^\vee$ denotes the Langlands dual group, and
 $\theta$ a quasi-split conjugation of $G^\vee$ so that $\eta$ is a
 corresponding involution of $G$.

Associated to $\theta$ is a finite collection  (possibly with multiplicities)
$\Theta(\theta)$ of conjugations 
of $\Gv$ all in the same inner class as $\theta$.
For each $\tau \in \Theta(\theta)$, we write $G^\vee_{\R,\tau}\subset
\Gv$ for the corresponding real form.
For each $[\lambda]\in \h/W\simeq (\h^{\vee})^*/W$, we write
$\HC_{\tau, [\lambda]}$ for the  dg category of Harish Chandra modules
for the real form $\Gv_{\R, \tau}$ with pro-completed generalized infinitesimal
character~$[\lambda]$.  
For simplicity, we will restrict our
attention to the case when $[\lambda]$ is regular.  

Fix a semisimple lift $\lambda \in \fg$ of the infinitesimal character
$[\lambda]$, and let $\alpha\in G$ denote the element
$\exp(\lambda)$. Let $G_\alpha\subset G$ be the reductive subgroup
that centralizes $\alpha$, and let $\cB_\alpha=G_\alpha/B_{\alpha}$ be
its flag variety. Consider the finite set of twisted conjugacy classes
$$
\Sigma({\eta,\alpha})=\{\sigma\in G|\sigma\eta(\sigma)=\alpha\}/G.
$$
Each $\sigma\in \Sigma({\eta,\alpha})$ defines an involution of
$G_\alpha$, and we write $K_{\alpha, \sigma}\subset G_\alpha$ for the
corresponding symmetric subgroup.

Soergel~\cite{Soergel} conjectures that there should be a Koszul
duality between categories of Harish Chandra modules and Langlands
parameters
$$
\xymatrix{
\bigoplus_{\tau\in \Theta(\theta)} {\HC}_{\tau,[\lambda]}
 \ar@{<->}[r]^-{?} &  \bigoplus_{\sigma\in \Sigma({\eta,\alpha})}
\D_{\coh}(K_{\alpha,\sigma}\backslash G_\alpha/B_{\alpha})
}
$$ respecting Hecke symmetries.  Soergel establishes this conjecture in
the case of tori, $SL_2$ and most importantly, for complex groups
$\Gv$ (considered as real forms of their complexifications). In the
complex case, the conjecture (for $\alpha=(e, e)$) reduces to the Langlands
duality for finite Hecke categories described above.

Soergel's conjecture is a categorical form of the results of Adams,
Barbasch and Vogan \cite{ABV}, who found an interpretation of the
Langlands parametrization of admissible representations of real groups
in terms of equivariant sheaves. These authors deduce the above
conjecture on the level of Grothendieck groups from Vogan's character
duality~\cite{Vogan}, and combine it with the microlocal geometry of
the cotangent bundles $T^*(K_{\alpha,\sigma}\backslash
G_\alpha/B_{\alpha})$ to study Arthur's conjectures.  As mentioned in
\cite{ABV}, it is important to find a way to fit together the spaces
$K_{\alpha,\sigma}\backslash G_\alpha/B_{\alpha}$, for varying
$\alpha$. In particular, this is necessary if one hopes to have a
uniform picture for representations with different infinitesimal
characters.

One of the outcomes of this paper is a solution to this problem in the
form of the Langlands parameter variety $\La^\eta$. The crucial change of
perspective is that loop spaces rather than cotangent bundles are a
more natural classical format for encoding the quantum geometry. 
The key insight is an identification
$$
\cL_\alpha/G \simeq  \coprod_{\sigma\in \Sigma({\eta,\alpha})}
\cL^u(K_{\alpha,\sigma}\backslash G_\alpha/B_{\alpha})
$$
proved in Corollary~\ref{final result}. In particular, as a unipotent loop space, $\cL_\alpha/G$ comes equipped
with a natural $\BS$-action
with resulting $\Gm$-action contracting to the fixed substack 
$$
X_\alpha/G =  \coprod_{\sigma\in \Sigma({\eta,\alpha})}
K_{\alpha,\sigma}\backslash G_\alpha/B_{\alpha}
$$

Now with
 the above preparations behind us,
we can state 
a more precise version of Theorem~\ref{cite categorical Langlands}.

\begin{thm}\label{categorical Langlands precise}
For any $\alpha\in H$, there is a canonical equivalence
$$\Coh_{[X_\alpha/G]}(\La^\eta_\alpha/G)^\BS_{loc} \simeq  \bigoplus_{\sigma\in \Sigma({\eta,\alpha})}
\D_\coh(K_{\alpha,\sigma}\backslash G_\alpha/B_{\alpha})
$$
between the  localized $\BS$-invariants  
and the Langlands parameter category.
\end{thm}

The proof of Theorem~\ref{categorical Langlands precise} involves a
detailed study of the interaction of the involution $\eta$ with the
equivalence of Theorem~\ref{first main result}.
A remarkable aspect of this argument is the appearance of the
complicated combinatorics of parameters for Harish Chandra modules
from nothing more than a formal construction involving the well
understood combinatorics of Weyl groups.


\subsection{Nilpotent cone and commuting variety}
We conclude with a final simple illustration of the general
theory of ~\cite{conns}. We will realize Rees modules 
(as recalled in Section~\ref{section: background} below) on the nilpotent cone in terms of
quasicoherent sheaves on the commuting variety. (With appropriate finiteness conditions,
one can similarly approach $\D$-modules as in the preceding sections.
But unlike the loop spaces of our previous examples, the commuting variety has
 nontrivial derived structure.
This means 
 the formulations are more complicated and we will only informally describe them.)

Let us first consider Rees modules on the adjoint quotient $G/G$ from the
perspective of loop spaces. As recalled in Section~\ref{section: background} below, we are thus led to consider the double
loop space 
$$\cL(G/G)=\cL(\cL {\rm B}G)\simeq \Map(T^2, {\rm B}G)\simeq \on{Loc}_G(T^2),
$$ the
derived stack of maps from the two-torus $T^2$ into ${\rm B}G$, or equivalently,
of $G$-local systems on $T^2$. By writing $T^2$ as a wedge of circles
with a disk attached, we can describe $\cL(G/G)$ as the derived stack
of pairs of commuting elements up to
conjugation 
$$\cL(G/G)\simeq \{(g_1, g_2) \in G\times G |\,
g_1g_2=g_2g_1\}/G.
$$ 
In other words, $\cL(G/G)$ is the derived fiber product modulo conjugation
$$
\cL(G/G) \simeq ((G\times G) \times_G \{e\})/G
$$
where the first map is the commutator $(g_1, g_2) \mapsto g_1g_2g^{-1}_1 g^{-1}_2$, and the second is
the inclusion of the identity $e\in G$.

In order to describe Rees modules on $G/G$, we need in fact only
consider formal loops in $G/G$, which are easily seen to be
$$
\Lhat(G/G)\simeq \{(g_1, g_2) \in G\times \wh G |\,
g_1g_2=g_2g_1\}/G.
$$ 
where $\wh G\subset G$ denotes the formal group.
 Observe that this is a kind of commuting variety where the second element is infinitesimally close to the identity.
 
 As an immediate consequence of results of~\cite{conns}
 recalled in Section~\ref{section: background} below,
 we have the following.
 
%

\begin{thm} 
There is a canonical equivalence 
$$
\cR_{G/G}\topmodule_\Z
\simeq 
\qcoh(\Lhat(G/G))^\BS
$$
between 
complete graded Rees modules on the adjoint quotient $G/G$
 and 
 $\BS$-invariant
 quasicoherent sheaves
  on the formal loop space $\Lhat(G/G)$.
\end{thm}

Now let us play with the theorem to find a more symmetric formulation.
We will replace the appearance of the entire group $G$ and the formal group $\wh G$ with the compromise
$G^u$ consisting of the formal neighborhood of unipotent elements. 

On the one hand, we can restrict to sheaves supported on $G^u \subset G$ on both sides of the theorem.
On the left hand side, we can consider the full subcategory 
$$
\cR_{G^u/G}\topmodule_\Z
\subset
\cR_{G/G}\topmodule_\Z
$$
of complete 
 graded Rees modules supported on $G^u/G$.
 On the right hand side, 
 we can consider the full subcategory 
 $$
\qcoh(\Lhat(G^u/G))^\BS
\subset
\qcoh(\Lhat(G/G))^\BS
$$
of
  $\BS$-invariant
 quasicoherent sheaves
  supported on $G^u/G$.

On the other hand, there is a close relationship between the formal loop space  $\Lhat(G/G)$ and the 
unipotent loop space 
$$
\cL^u(G/G)\simeq \{(g_1, g_2) \in G\times G^u |\,
g_1g_2=g_2g_1\}/G
$$ 
Namely, we have an inclusion $\Lhat (G/G)\subset \cL^u(G/G)$ which is a base change of the natural inclusion $\wh G \subset G^u$.
This provides a
restriction  map
$$
\xymatrix{
\qcoh(\cL^u(G/G))^\BS
\ar[r] &
\qcoh(\Lhat(G/G))^\BS
}
$$
which fails to be an equivalence only because there is a kernel consisting of sheaves which are very far from finitely generated. 

Now let us introduce the {\em derived unipotent commuting variety} 
$$
\cC^u= \{(g_1, g_2) \in
   G^u\times G^u |\, g_1g_2=g_2g_1\}
   $$
   where as usual  $G^u/G$ denotes the formal neighborhood of the unipotent cone.
   Its quotient by the natural $G$-action by simultaneous conjugation is the double unipotent loop space 
   $$
   \cC^u/G=\cL^u(\cL^u({\rm B}G))
   $$ 
   which naturally
classifies $G$-local systems on the two-torus $T^2$ with unipotent
monodromies. 

Observe that $\cC^u/G$ carries two commuting
$\Gm$-actions given by rescaling the two unipotent group
elements. Moreover, its realization as double loops (or as local
systems) makes apparent two commuting $S^1$-actions.  
Altogether, we
have an action of the product $\BS\times \BS$,
but we will restrict our attention to
the second factor of $\BS$. 

Putting together our previous observations, the above theorem provides a natural functor
$$
\xymatrix{
\qcoh(\cC^u/G)^\BS
\ar[r] &
\cR_{G^u/G}\topmodule_\Z
}$$
from $\BS$-invariant quasicoherent sheaves on $\cC^u/G$ to 
   complete 
 graded Rees modules supported on $G^u/G$.
Furthermore, if we restrict to appropriately finite sheaves, the functor becomes an equivalence.
Going one step further and localizing, we obtain an equivalence between appropriately
finite sheaves and  $\D$-modules on the nilpotent cone $\cN \subset \fg$.

This connects the category at the heart of Springer theory, with another category of
great current interest, see in particular the recent work of
Schiffmann and Vasserot \cite{SV1,SV2} and Ginzburg \cite{Ginzburg}.
The commuting variety (in the case of $GL_n$) is closely related to
the Hilbert scheme of points in $\C^n$, and its Grothendieck group of
$\Gm\times\Gm$-equivariant coherent sheaves is linked to the theory of
Macdonald polynomials and double affine Hecke algebras\footnote{While
  the authors cited above consider commuting elements in the Lie
  algebra rather than in the formal neighborhood $G^u$ of the unipotent cone, their
  $\Gm\times\Gm$-equivariant $K$-groups are equivalent.}.  In
particular, the interpretation via local systems on the torus
identifies this $K$-group (via the geometric Langlands conjecture)
with the elliptic Hall algebra and the theory of Eisenstein series.

\begin{remark}
  It would be interesting to consider the fully invariant version
  $\qcoh(\cC^u/G)^{\BS \times \BS}$ and various localizations. Any
  construction respecting the symmetries of the torus can be expected
  to carry a canonical $SL_2(\Z)$-action.
\end{remark}


\section{Loop spaces and flat connections}\label{section: background}



We review here some of the relevant results of \cite{conns}. It is
written in the language of derived algebraic geometry over a fixed
$\Q$-algebra $k$, and in this paper, we will always work over $\C$.
As it turns out, the geometric objects which arise in this paper (with
the exception of the commuting variety) are underived stacks (in fact,
quotients of quasi-projective varieties or formal schemes by reductive
groups), but we will make substantial use of results which depend on
the flexible context of derived stacks.

In broad outline, to arrive at derived algebraic geometry from
classical algebraic geometry, one takes the following steps. First, one replaces
commutative $k$-algebras by {\em derived $k$-algebras}, namely simplicial (or
equivalently, for $k$ a $\Q$-algebra, connective differential graded) commutative
$k$-algebras. This allows for derived intersections by replacing the
tensor product of rings with its derived functor.  Second, one
considers functors of points on derived rings that take values not
only in sets but in simplicial sets or equivalently topological
spaces. This allows for derived quotients by keeping track of gluings
in the enriched theory of spaces.

Roughly speaking, a {\em derived stack} is a functor from derived
rings to topological spaces satisfying an sheaf axiom with respect to
the \'etale topology on derived rings.
We will denote by $\DSt$ the
$\oo$-category of derived stacks over (the fppf site of) $\C$.
 Examples
include all of the schemes of classical algebraic geometry, stacks of
modern algebraic geometry, along with topological spaces in the form
of locally constant stacks.  Any derived stack has an ``underlying''
underived stack, obtained by restricting its functor of points to
ordinary rings. In fact, derived stacks can be viewed as
nil-thickenings of underived stacks, just as supermanifolds are
infinitesimal thickenings of manifolds.

One of the complicated aspects of the theory is that the domain of
derived rings, target of topological spaces, and functors of points
themselves must be treated with the correct enriched homotopical
understanding.  We recommend To\"en's excellent survey \cite{Toen} (as
well as \cite{conns}) for more details and references.  It was our
introduction to many of the notions of derived algebraic geometry, in
particular derived loop spaces.


\subsection{Three kinds of loops}


We consider the circle $S^1={\rm B}\Z=K(\Z,1)$ as a (locally constant)
stack.  Since we work in characteristic zero, its algebra of cochains
is formal
$$\cO(S^1) = C^*(S^1,k)\simeq H^*(S^1,k)=k[\eta]/(\eta^2),\quad
|\eta|=1.$$ In fact, $\cO(S^1)$ is the free symmetric algebra (in the
graded sense) on a single generator of degree one.  Thus its
affinization (the functor corepresented by its algebra of cochains) is
an odd version of the affine line, namely the classifying stack of the
additive group
$$\Aff(S^1)\simeq {\rm B}\Ga\simeq K(\Ga,1)\simeq \aline[1].$$ This is
the stack that assigns to a $k$-algebra $R$ the space ${\rm B}R$, or
equivalently (in the case of $R$ discrete) the Eilenberg-MacLane space
$K(R,1)$, where we consider $R$ as an additive group.  We will prefer
the notation ${\rm B}\Ga$ to emphasize the group structure on
$\Aff(S^1)$.

\begin{defn} The {\em loop space} of a derived stack $X$ is the derived
  mapping stack $$\cL X=\Map_{\DSt}(S^1,X).$$

The {\em unipotent loop space} of $X$ is the
  derived mapping stack 
  $$\cL^u X=\Map_{\DSt}(\Aff(S^1),X).
  $$ 
  
 The {\em formal loop space} of  $X$ is the formal completion
 of $\cL X$ along the constant loops 
 $$\Lhat X=\widehat{\cL X}_X.
 $$
 \end{defn}

Pulling back along the affinization homomorphism $S^1\to \Aff(S^1)$ defines a
morphism $\cL^u X\to \cL X$. This is an equivalence when $X$ is a
derived scheme (quasi-compact with affine diagonal). More generally,
when $X$ is a geometric stack (Artin with affine diagonal), formal
loops are unipotent: the morphism 
factors
$$\Lhat X \to \cL^u X \to \cL X.$$

\subsection{The basic example}\label{basic example}
It will be useful throughout the paper to keep the following
perspective in mind. Suppose $X$ classifies some kind of object in the
sense that maps $S\to X$ form the space of such objects over $S$. Then
the loop space $\cL X$ classifies families of such objects over the
circle $S^1$. (It is worth emphasizing that a family of objects over a
topological space is by necessity locally constant.)  Likewise, the
formal loop space $\Lhat X$ classifies families over $S^1$ that are
infinitesimally close to a trivial family, and the unipotent loop
space $\cL^u X$ classifies algebraic one-parameter families of
automorphisms.

We illustrate the above notions with the basic motivating example
$X={\rm B}G = pt/G$, the classifying stack of a reductive group $G$, so
$X$ classifies principal $G$-bundles. We have the elementary
identifications of loop spaces
$$
\cL ({\rm B}G) \simeq G/G
\qquad
\cL^u ({\rm B}G) \simeq G^u/G
\qquad
\Lhat ({\rm B}G) \simeq \wh G/G
$$ where all of the quotients are with respect to conjugation, $G^u$
is the formal neighborhood of the unipotent elements of $G$, and $\wh
G$ is the formal group of $G$.  In other words, $\cL {\rm B}G$ classifies
$G$-local systems on $S^1$. If
we trivialize such a local system at a point, its monodromy gives an
element of $G$. Forgetting the trivialization passes to the adjoint
quotient $G/G$.  Likewise, $\Lhat {\rm B}G$ classifies local systems whose
monodromy is in the formal neighborhood of the identity, and $\cL^u
{\rm B}G$ classifies local systems with unipotent monodromy.

Recall the characteristic polynomial map 
$$
\xymatrix{
\chi:\cL ({\rm B}G) \simeq G/G \ar[r] & H\quot W
}
$$
where $H\quot W = \Spec \cO(H)^W$ denotes the affine quotient.
By Chevalley's Theorem, $\chi$ induces an equivalence on global functions 
$$
\xymatrix{
\chi^*:\cO(H)^W \ar[r]^-\sim & \cO(G/G)
}
$$ and thus $H\quot W$ is the affinization of $\cL ({\rm B}G) \simeq G/G$.
From the point of view of local systems, the characteristic polynomial
map simply remembers the eigenvalues of the monodromy.

Observe that the identification $\cL ({\rm B}G)\simeq G/G$ restricts to an
identification
$$\cL^u ({\rm B}G) \simeq \chi^{-1}(\wh H\quot W)
$$ where $\wh H$ denotes the formal group of $H$.  More generally, fix
a semisimple $\alpha\in G$ with class $[\alpha]\in H\quot W$, and let 
 $G(\alpha)\subset
G$ be the centralizer of $\alpha$. Then we have an identitication
$$\cL^u({\rm B}G(\alpha))\simeq \chi^{-1}(\wh H_{W\cdot\alpha}\quot W)$$ 
where $\wh H_{W\cdot\alpha}$ denotes the formal neighborhood of the Weyl orbit $W \cdot \alpha\subset H$.

\begin{remark}[Jordan decomposition]

We can interpret the above identifications in terms of
{\em Jordan decomposition} of elements of $G$. Namely, separating out the semisimple part of elements breaks up the loop space $\cL {\rm B}G$ into unipotent
loop spaces. 

One can imitate the Jordan decomposition for the loop space $\cL X$ of
a general geometric stack. We will perform this decomposition
explicitly for examples associated with flag varieties in Sections
\ref{Steinberg section} and \ref{Langlands section}. We merely note
here that the underlying underived stack of $\cL X$, the inertia stack
$IX$, is an affine group scheme over $X$, and so we may speak of the
semisimple part of any loop $\gamma\in \cL X$ as a point of inertia
$\gamma_{ss}\in IX$. One may then combine this with the notion of
unipotent loop introduced above and extend the Jordan decomposition of
inertia to a decomposition of the loop space in terms of unipotent
loop spaces of a collection of associated stacks. Combined with the
relation of sheaves on unipotent loops to $\D$-modules recalled below,
this provides an approach to understanding sheaves on the entire loop
space.
\end{remark}

\subsection{Rotating loops}\label{reasonable}

The circle $S^1$, and hence the loop space $\cL X$, carries a natural
rotational $S^1$-action.  The affinization $\Aff(S^1) \simeq {\rm B}\Ga$ of the circle, and
hence the unipotent loop space $\cL^u X$, carries a natural
translational ${\rm B}\Ga$-action so that the inclusion $\cL^u X\to \cL X$
is equivariant for the natural homomorphism $S^1 \to \Aff(S^1) \simeq {\rm B}\Ga$. In addition, $\cL^u X$ carries a compatible action
of the multiplicative group $\Gm$ induced by its action on $\Ga$. This
$\Gm$-action is an expression of the formality of
$\cO(S^1)=C^*(S^1,k).$

It is shown in \cite{conns} that the translational  ${\rm B}\Ga$-action on $\cL^u X$ is completely determined by the data of a degree one vector
field with square zero. Restricted to formal loops, this vector field is identified with
the de Rham differential. Let $\BT_X[-1] =\Spec_{\cO_X}
\Omega_X^{-\bullet}$ denote the odd tangent bundle of $X$, and $\wh
\BT_X[-1] $ its formal completion along the zero section.

\begin{thm}[\cite{conns}]
Let $X$ be a geometric stack. There is a canonical identification
$$ \Lhat X \simeq \wh \BT_X[-1]
$$
such that the odd vector field of the $S^1$-action of loop rotation corresponds
to that given by the de Rham differential.
\end{thm}

\begin{remark}
The theorem is a generalization to arbitrary
geometric stacks of the exponential map identifying the adjoint
quotients of the completed Lie algebra $\wh{\g}/G$ and formal group
$\wh{G}/G$.
\end{remark}

The above theorem provides a relation between equivariant sheaves on the
loop space of a smooth stack and sheaves with flat connection on the stack itself. 

Consider the equivariant cohomology ring 
$$
\cO({\rm B}S^1) \simeq H^*({\rm B}S^1, k)\simeq k[u], \mbox{ with $u$ of degree $2$}.
$$  
The affinization morphism $S^1\to {\rm B}\Ga$ induces an identification $\cO({\rm B}S^1) \simeq \cO({\rm B}{\rm B}\Ga)$, 
and hence there is a natural $\Gm$-action
on the cohomology such that $u$ has weight $1$.

Define the graded Rees algebra $\cR_{X}$ to be the sheaf of $\Z$-graded $k$-algebras
$$
\cR_{X}=\bigoplus_{i\geq 0} u^i \D_{X}^{\leq i} \subset \cD_X \otimes_k k[u]
$$ where $\cD_X^{\leq i} \subset \cD_X$ denotes differential operators
of order at most $i$.  Note that weights and cohomological degrees in
$\cR_X$ are related ($\cR_X$ is concentrated in positive cohomological
degrees, which are twice the weight degrees).

Let $\cR_X\topmodule_\Z$ denote the stable $\oo$-category of graded
quasicoherent sheaves on $X$ equipped with the compatible structure of
graded complete $\cR_X$-module.

\begin{thm}[\cite{conns}]
For $X$ a smooth geometric stack, there is a canonical equivalence of stable $\oo$-categories
$$ \qcoh(\Lhat X)^{\BS}
\simeq\cR_{X}\topmodule_\Z. 
$$
\end{thm}

\begin{remark}
The classical Rees algebra $\cR^{cl}_X$ is positively graded by weight but lives in cohomological degree zero.
As explained in \cite{conns}, 
there is a canonical shear equivalence
$$
\cR^{cl}_{X}\topmodule_\Z\simeq \cR_{X}\topmodule_\Z.
$$
\end{remark}

To recover $\D_X$-modules from $\cR_X$-modules, we must localize. On the one hand, any
category of $S^1$-equivariant (or ${\rm B}\Ga$-equivariant) sheaves such as $ \qcoh(\Lhat X)^{\BS}$ is linear over the
equivariant cohomology ring 
$
\cO({\rm B}S^1)\simeq \cO({\rm B}{\rm B}\Ga) \simeq  k[u].
$
On the
other hand, by construction, the algebra $k[u]$ maps to the graded Rees algebra
$\cR_X$ as a central subalgebra.
 
 When we invert $u$ and localize the categories, we will denote the
 result by the subscript ${loc}$.  
So for a compatibly graded $k[u]$-linear stable $\oo$-category $\cC$, we have the
new stable $\oo$-category $\cC_{loc}$ which can be identified with 
$\cC \ot_{k[u]}k[u,u\inv].
$
If $\cC$ were not graded, $\cC_{loc}$ would be only
 $\Z/2\Z$-graded since $u$ is of 
 cohomological degree $2$. But when we localize a compatibly graded category, the localized category maintains a usual cohomological degree.
The natural projection $p:\cC\to \cC_{loc}$
satisfies 
$$p(c[-2]\langle 1\rangle) = p(c)
$$
where $[-2]$ denotes the cohomological shift by two, and $\langle 1\rangle$ denotes the weight shift by one.

Now to describe a localized version of the equivalence of the previous theorem, we introduce the following further notions.
First we assume that $X$ is smooth. Further, we will assume $X$ has the property that its loop space $\cL X$ is in fact underived. 
As explained in Proposition~\ref{dg structure of loops} below, for the quotient of a quasi-projective variety
by an affine group scheme,
this is equivalent to there being finitely many orbits.

We will write $\Coh_{[X]}(\cL^u X)^\BS$ for the dg category of $\BS$-equivariant coherent sheaves on 
the unipotent loop space $\cL^u X$
whose restriction to $X$ is coherent (or
equivalently,  perfect since $X$ is smooth).
We refer to 
$\Coh_{[X]}(\cL^u X)^\BS$ as the dg category of $\BS$-equivariant coherent sheaves on
 $\cL^u X$ coherent along  $X$.

In the rest of the paper, the only result of \cite{conns} we will apply is the following.

\begin{corollary}[\cite{conns}]
Assume $X$ is a smooth geometric stack with underived loop space.
There is a canonical equivalence
 of stable $\oo$-categories 
$$
\Coh_{[X]}(\cL^u X)^{\BS}_{loc}
   \simeq \cD_\coh(X).
$$
\end{corollary}

Let us sketch the reduction of the above assertion to a local statement. Let $U\to X$ be a smooth morphism from a smooth affine scheme. By construction, the dg category $\cD_\coh(X)$ is equivalent to  the limit of the natural diagram of dg categories that assigns $ \D_\coh(U)$ to any such map $U\to X$. 
  By the results of \cite{conns} and Koszul duality, 
this diagram  is in turn equivalent to the natural diagram that assigns $ \Coh_{/U}(\cL U)^{\BS}_{loc}$ to any such map
 $U\to X$. 
 Here we write $\Coh_{/U}(\cL U)$ for the dg category of quasicoherent sheaves on  $\cL U$ that are coherent (or equivalently, perfect since $X$ is smooth) over $U$.
 
Now to arrive at the above corollary, we need only check that  $\Coh_{[X]}(\cL^u X)^{\BS}_{loc}$
is equivalent to the limit $\lim_U \Coh_{/U}(\cL U)^{\BS}_{loc}$.

By the results of \cite{conns}, we know that $\qcoh(\Lhat X)$ is
equivalent to the limit $\lim_U \qcoh(\cL U)$ compatibly with all
symmetries. Thus it remains to check that this equivalence restricts
to an equivalence from $\Coh_{[X]}(\cL^u X)^\Gm$ to the limit $\lim_U
\Coh_{/U}(\cL U)^{\Gm}$. This comes down to two elementary statements.

First, one checks that the natural restriction map $\Coh(\cL^u X)^\Gm
\to \Coh(\Lhat X)^\Gm$ on all $\Gm$-equivariant coherent sheaves is an
equivalence. Here the simple but key observation is that $\Lhat X$ is
the completion of $\cL^u X$ along the natural closed embedding
$X\hookrightarrow \cL^u X$, and the $\Gm$-action contracts $\cL^u X$
back onto $X$. A more precise formulation of the last observation is
that the $\Gm$-weights in the functions $\cO_{\cL^u X}$ are
non-positive, and the ideal $\cI_X \subset \cO_{\cL^u X}$ consists
precisely of functions with strictly negative weights. With this in
hand, it is not difficult to check the analogous and sufficient local
statement about the restriction of finitely generated graded modules
over a graded ring to such a completion.

Second, applying the preceding, we are left to check  that a $\Gm$-equivariant coherent sheaf on $\cL^u X$ restricts along $X\to \cL^u X$ to a coherent sheaf if and only if its pullback along $\cL U \to \cL^u X$ is coherent over $U$. For this assertion, one need only observe that the map $\cL U \to \cL^u X$ factors through the canonical maps  $\cL U \to U \to X \to \cL^u X$ since $\cL^u X$ is underived (by assumption) and $\cL U$ is a derived thickening of $U$ (since $U$ is a scheme).



\section{Steinberg varieties as loop spaces}\label{Steinberg section}

In this section, we explain how one can view equivariant Steinberg
varieties as loop spaces.  The entire equivariant
Grothendieck-Steinberg variety is a loop space, and its
fixed-monodromy subspaces are unipotent loop spaces.

\subsection{The Grothendieck-Springer resolution via loops}

We begin with a warm-up, identifying the Grothendieck-Springer
resolution in terms of loop spaces. We follow the blueprint of our
basic example ${\rm B}G$, discussed in Section \ref{basic example} above.

Let ${\rm B}B = pt/B$ denote the classifying stack of a Borel subgroup $B \subset G$.
We have the elementary identifications of loop spaces
$$
\cL ({\rm B}B) \simeq B/B
\qquad
\cL^u ({\rm B}B) \simeq B^u/B
\qquad
\Lhat ({\rm B}B) \simeq \wh B/B
$$

Recall the Grothendieck-Springer simultaneous resolution 
$$
\xymatrix{
\cB & \ar[l]_-\pi \wt G = \{(g, B) \in G\times \cB\, |\, g\in B\} \ar[r]^-\mu & G.
}
$$
We have the ordered eigenvalue map 
$$
\xymatrix{
\tilde \chi: \wt G \ar[r] & H
& 
\tilde\chi (g, B) = [g]\in B/U \simeq H.
}
$$ The fiber $ \tilde \chi^{-1}(e)\subset \wt G$ over the identity
$e\in H$ is the traditional Springer resolution where the group
element $g$ is required to be unipotent.

Observe that there is a natural identification 
$$
\cL ({\rm B}B) \simeq B/B \simeq \wt G/G 
$$ 
where the quotients are with respect to conjugation.
Since the ordered eigenvalue map $\tilde \chi$ is invariant under conjugation,  it descends to 
 to the adjoint quotient 
$$
\xymatrix{
\tilde \chi: \wt G/G \ar[r] & H
}
$$ 
This map in turn induces an equivalence on global functions 
$$
\xymatrix{
\tilde \chi^*:\cO(H) \ar[r]^-\sim & \cO(\wt G/G)
}
$$
and thus $H$ is the affinization of $\cL ({\rm B}B)\simeq \wt G/G$.

\begin{remark}
We can view $\cL({\rm B}B) \simeq B/B$ as $B$-local systems on the circle
$S^1$. Of course, it is obviously equivalent to $G$-local systems with
a $B$-reduction. Alternatively, it is also equivalent to $G$-local
systems on $S^1$ with a monodromy-invariant $B$-reduction at a point
of $S^1$.  If we trivialize the $G$-local system at the point, we
obtain an element of $\wt G$.
\end{remark}


\subsection{Loops in flag varieties}


We will use the name Grothendieck-Steinberg variety $\St_{}$ for the
fiber product
$$
\St =\wt G\times_G \wt G = \{(g, B_1, B_2) \in G\times \cB\times\cB\, |\, g\in B_1\cap B_2\}.
$$
 In general,
$\St_{}$ is connected, but has irreducible components  labeled
by the Weyl group $W$.  Hence if $G$ is noncommutative,
$\St_{}$ is singular.

The following is the fundamental observation that motivated this paper.

\begin{thm} \label{steinberg variety as loop space}
We have a canonical identification of the equivariant Steinberg variety as a loop space
$$
\cL (B\bs G/B) \simeq \St/G.
$$ 
Moreover, the Steinberg variety is underived.
\end{thm}

\begin{proof}
We combine the standard identifications 
$$B\bs G/B\simeq \cB \times \cB/G\simeq {\rm B}B \times_{{\rm B}G} {\rm B} B$$ with the 
loop space identification $$\wt{G}/G\simeq B/B\simeq \cL({\rm B} B).$$ Since loops and fiber products commute (both are limits), we calculate the loop space of $B\bs G/B$ as the fiber product of loop spaces 
$$\cL (B\bs G/B)
\simeq\cL({\rm B} B)\times_{\cL({\rm B} G)} \cL({\rm B} B)\simeq \wt{G}/G \times_{G/G} \wt{G}/G.$$
Finally, consider the universal morphism (associated to the fiber product on the right)
$$
\xymatrix{
 \St/G = (\wt{G} \times_{G} \wt{G})/G \ar[r] & \wt{G}/G \times_{G/G} \wt{G}/G 
}
$$ 
It is an equivalence since its base change over $pt\to {\rm B}G$ is an equivalence (since the fiber product on the right distributes over the base change).

The fact that $\St/G$ is underived (the coincidence of derived and naive fiber products) follows from the fact that the Grothendieck-Springer resolution $\mu:\wt G\to G$ is a semi-small resolution. Alternatively, we may deduce it from Proposition~\ref{dg structure of loops} immediately following.
\end{proof}

\begin{prop}\label{dg structure of loops}
  Suppose $X = Y/G$, where $Y$ is a quasi-projective variety and $G$ is
  affine. Then the loop space $\cL X$ has trivial derived structure if
  and only if there are finitely many $G$-orbits in $Y$.
  \end{prop}

\begin{proof}
  We represent $X$ as a groupoid scheme with $X_0=Y$ as scheme of
  objects (0-simplices) and $X_1=G\times Y$ as scheme of morphisms (1-simplices).
  It is convenient to rewrite the pushforward to $X$ of the structure
  sheaf $\cO_{\cL X}$ as the descent of the derived tensor product
$$
\cO_{X_1} \otimes^\bbL_{X_0\times X_1} \cO_{X_1}
$$
where $X_1$ maps to $X_0\times X_1$ by the product maps $\ell\times \id_{X_1}$
and $r\times \id_{X_1}$.
This can be viewed as the structure sheaf of the derived intersection of the subschemes
$\Gamma_\ell,\Gamma_r\subset X_0\times X_1$ given by the graphs of $\ell, r$ respectively.
Here loops are thought of as pairs of $1$-simplices that are equal
and such that the left end of the first is glued to the right end of the second.

Now our assertion will follow from a simple dimension count.
Let $n_0$ and $n_1$ be the dimensions of $X_0$ and $X_1$ respectively.
On the one hand,
the expected dimension of the intersection $\Gamma_\ell\cap \Gamma_r$ inside of $X_0\times X_1$
is given by $n_1+n_1 - (n_0 + n_1) = n_1 - n_0$. On the other hand,
each isomorphism class of objects of $X$
contributes a subscheme of precisely dimension $n_1-n_0$ to the intersection.
Thus the intersection has the expected dimension if and only if there is no nontrivial moduli
of isomorphism classes of objects.
\end{proof}
%


\begin{remark}
We can view $\BGB \simeq {\rm B}B \times_{{\rm B}G} {\rm B}B$ as the moduli of $G$-local
systems on the interval $[0,1]$ with $B$-reductions at the end points
$\{0,1\}$. Similarly, we can view the equivariant
Grothendieck-Steinberg variety $ \cL (B\bs G/B) \simeq \St/G $ as the
moduli of $G$-local systems on the cylinder $Cyl = [0,1] \times S^1$
with $B$-reductions at the boundary circles $\{0,1\} \times S^1$. \end{remark}


\subsection{Fixed monodromy}

We have two copies of the ordered eigenvalue map 
$$ \xymatrix{ \tilde \chi: \St \ar[r] & H \times H & \tilde\chi (g,
  B_1, B_2) = ([g]_1, [g]_2)\in B_1/U_1 \times B_2/U_2 \simeq H \times
  H.  }
$$ The fiber $\tilde \chi^{-1}(e, e)\subset \St $ over the identity
$(e, e)\in H\times H$ is the usual Steinberg variety, where the group
element $g$ is required to be unipotent.

The image of $\tilde \chi$ consists of pairs of elements $(\alpha, \beta) \in H\times H$ related by the
Weyl group action:
$\beta = w\cdot\alpha$, for some $w\in W$. 
In other words, it consists of the union 
$$
\cH = \coprod_{w\in W} \{ (h, w\cdot h) | h\in H\}\subset H\times H
$$ 
of the graphs of the automorphisms of $H$
given by Weyl group elements.

Since the ordered eigenvalue map $\tilde\chi$ is invariant under
conjugation, it descends to the adjoint quotient
$$
\xymatrix{
\tilde \chi: \St/G \ar[r] & \cH.
}
$$ 
This map in turn
realizes $\cH$ as the affinization
of $\St/G$.


Fix an element $(\alpha, w\cdot\alpha) \in \cH$, and let $\wh \cH_{\alpha, w\cdot \alpha}$ be its formal neighborhood.
We will use the name monodromic Steinberg variety  for the inverse image
$$
\St_{\alpha, w\cdot\alpha} = \tilde\chi^{-1}(\wh \cH_{\alpha, w\cdot \alpha})
$$

%

Let $\cO_\alpha\subset G$ denote the semisimple conjugacy class
corresponding to $\alpha$. Fix once and for all an element $\tilde
\alpha\in\cO_\alpha$, and let $G(\alpha)\subset G$ denote its
centralizer. In general, $G(\alpha)$ is reductive of the same rank as $G$, and often turns
out to be a Levi subgroup. 

We will affix the symbol $(\alpha)$ to our usual notation when
referring to objects associated to $G(\alpha)$. 
So for example, we
write $B(\alpha)\subset G(\alpha)$ for a Borel subgroup.

The following key variation on Theorem~\ref{steinberg variety as loop
  space} further justifies the importance of loop spaces to Langlands
parameters for representations.

\begin{thm} \label{monodromic steinberg variety as loop space} For any
  $\alpha\in H$, and $w\in W$, we have a canonical identification of
  the unipotent loop space and monodromic Steinberg variety
  $$
  \cL^u (B(\alpha)\bs G(\alpha)/B(\alpha)) \simeq \St_{\alpha, w\alpha}/G
  $$
\end{thm}

\begin{proof}
  First, note that multiplication by the central element $\tilde \alpha\in G(\alpha)$ provides an equivalence
  $$
\xymatrix{
 \cL^u (B(\alpha)\bs G(\alpha)/B(\alpha)) \simeq \St(\alpha)_{e,e}\ar[r]^-\sim & \St(\alpha)_{\alpha,\alpha}
}  $$
  where $e\in H$ is the identity.
  
Now for each $w\in W$, one can check that there is a map
$$
\xymatrix{
\St(\alpha)_{\alpha,\alpha}\ar[r] & \St_{\alpha,w\alpha}
&
(g, B(\alpha)_1, B(\alpha)_2) \ar@{|->}[r] & (g, B_1, B_2)
}
$$
uniquely characterized by the properties:
$$
B_1\cap G(\alpha) = B(\alpha)_1 \quad B_2\cap G(\alpha) =
B(\alpha)_2
$$
$$
[g]_1 \in \wh H_\alpha \quad [g]_2 \in \wh H_{w\alpha}
$$
where $[g]_1,[g]_2$ denote the classes of $g$ in $B_1/U_1$,
$B_2/U_2$ respectively.

Passing to the respective quotients gives the sought-after
isomorphism.
\end{proof}

%

An interesting aspect of the theorem is the general ``discontinuity''
of the objects appearing on the left hand side. From a geometric
perspective, the quotients $B(\alpha)\backslash G(\alpha)/B(\alpha)$ do not
form a nice family as we vary the parameter $\alpha$.
But the theorem says that the loop spaces of these quotients do fit
into the nice family $\tilde \chi:\St/G\to \cH\subset H\times H$. Here
we should emphasize that we are thinking about $\St/G$ as a loop
space, rather than along the more traditional lines identifying the
usual Steinberg variety $\wt{\chi}^{-1}(e,e)$ with the union of
conormals to $B$-orbits in $\cB$. In the discussion to follow, we describe similar results for
geometric parameter spaces for real reductive groups. In that context,
it is only the loop spaces that fit together into a nice family, not
the cotangent bundles.

\subsection{Application to Hecke categories}\label{application to Hecke}

Let us focus on the most interesting case of trivial monodromy
when $\alpha$ is the identity $e\in H$.  By Theorem~\ref{monodromic
  steinberg variety as loop space}, we have a canonical identification
  $$
  \cL^u (B\bs G/B) \simeq \St_{e, e}/G
  $$ where the Steinberg variety $\St^u=\St_{e, e}$ is nothing
  more than the fiber product
  $$
\St^u =\wt G^u\times_{G^u} \wt G^u = 
\{(g, B_1, B_2) \in G^u\times \cB\times\cB\, |\, g\in B_1\cap B_2\},
$$
where $G^u\subset G$ is the formal neighborhood of the unipotent elements, and $\wt G^u$ is its Springer resolution 
$$
 \wt G^u = \{(g, B) \in G^u\times \cB\, |\, g\in B\}.
 $$
 
Now our results from \cite{conns}, as recalled in 
Section~\ref{section: background}, immediately provide
    a canonical equivalence of $\oo$-categories
$$
\Coh_{[(\cB \times \cB)/G]}(\St^u/G)^{\BS}_{loc} \simeq \D_{\coh}(B\bs G/B).
$$
Observe that on the one hand, the underlying dg category
$\Coh_{[(\cB \times \cB)/G]}(\St^u/G)^{\Gm}$ is our version of the affine Hecke category $\cH^{\it{aff}}_G$. On
the other hand, the dg category $\D_{\coh}(B\bs G/B)$ is the finite Hecke
category $\cH_G$. Thus we arrive at the following fundamental relationship.

\begin{corollary}\label{steinberg main cor}
 The localized $\BS$-invariants with bounded above weight in the affine Hecke category
  $\cH^{\it{aff}}_G$ are canonically equivalent to the finite Hecke
  category $\cH_G$.
\end{corollary}

For general monodromy $\alpha\in H$,
again by Theorem~\ref{monodromic steinberg variety as loop space},
we have a canonical identification
   $$
  \cL^u (B(\alpha)\bs G(\alpha)/B(\alpha)) \simeq \St_{\alpha, w\alpha}/G
  $$ Thus we can transport the canonical $\BS$-action on
  the unipotent loop space of the left hand side to the monodromic
  Steinberg variety of the right hand side.

\begin{corollary}\label{steinberg general monodromy}
For general $\alpha\in H$, we have a canonical equivalence
$$
\Coh_{[B(\alpha)\bs
G(\alpha)/B(\alpha)]}(\St_{\alpha, w\alpha}/G)^{\BS}_{loc} \simeq \D_{\coh}(B(\alpha)\bs
G(\alpha)/B(\alpha)).
$$
\end{corollary}

\begin{remark}[Comparison of circle actions]\label{comparison of actions}
In Corollary~\ref{steinberg general monodromy}, we use an $\BS$-action, and in particular an $S^1$-action, on 
 $\St_{\alpha,
   w\alpha}/G$ coming from its identification as a unipotent loop
 space
  $$
\St_{\alpha, w\alpha}/G\simeq   \cL^u (B(\alpha)\bs G(\alpha)/B(\alpha)).
  $$
On the other hand, by Theorem~\ref{steinberg variety as loop space}, we have a canonical
identification
$$
\cL (B\bs G/B) \simeq \St/G,
$$
and it is evident that the rotational $S^1$-action on the loop space of
the left hand side restricts to an $S^1$-action on the substack
 $$
 \St_{\alpha, w\alpha}/G \subset \St/G.
 $$ 
To distinguish the two $S^1$-actions, we will write $\rho$ for the former action and $\rho_\alpha$ for the latter.

 Our aim here is to compare the equivariant localization with respect to $\rho$ described in Corollary~\ref{steinberg general monodromy} with that with respect to $\rho_\alpha$.
  We will see that the two $S^1$-actions are related in a simple way that leads to an equivalence of their equivariant localizations.  

  First, observe that there is a canonical $S^1$-action $\varphi_\alpha$ on the stack
$$
\xymatrix{
B(\alpha)\bs G(\alpha)/B(\alpha) \simeq {\rm B} B(\alpha) \times_ {{\rm B} G(\alpha)}  {\rm B} B(\alpha)
}
$$
given by the simultaneous universal action of the central element $\alpha$   on each of the factors of the fiber product.
From a classifying bundle point of view, $\varphi_\alpha$  acts on any $G(\alpha)$-bundle
by the central element $\alpha$, and likewise on  any $B(\alpha)$-bundle.

Now given any stack $X$ with an $S^1$-action $\tau$, its loop space $\cL X$ comes equipped with an $S^1\times S^1$-action determined on the factors by the usual loop rotation action $\rho$ and the action $\cL\tau$ on loops   induced by the original action $\tau$.
   Similarly, 
    the unipotent loop space $\cL^u X$ comes equipped with an $\BS\times S^1$-action determined on the factors by the given actions. (These considerations are very general: independent symmetries of stacks $D$ and $X$ provide canonically commuting symmetries of the mapping stack $\Map_{\DSt}(D, X)$.)

Applying this to 
the canonical $S^1$-action $\varphi_\alpha$  on $B(\alpha)\bs G(\alpha)/B(\alpha) $, we obtain an $\BS\times S^1$-action,
and in particular an $S^1\times S^1$-action
$\rho\times\cL\varphi_\alpha$,
on the unipotent loop space
 $$
   \xymatrix{
 \cL^u (B(\alpha)\bs G(\alpha)/B(\alpha))   \simeq \St_{\alpha, w\alpha}/G}.
  $$

Tracing through the various constructions,
we find that
the  diagonal $S^1$-action associated to the product $S^1\times S^1$-action $\rho \times \cL\varphi_\alpha$ is equivalent to the $S^1$-action $\rho_\alpha$. 
This can be viewed as a factorization
of the $S^1$-action $\rho_\alpha$ into its semisimple part $\cL\varphi_\alpha$ and unipotent part $\rho$.
To  perform equivariant localization with respect to $\rho_\alpha$,
   it suffices to first perform equivariant localization  with respect to $\rho$ then study the induced action by $\cL\varphi_\alpha$.
   More precisely, we claim that the induced action by $\cL\varphi_\alpha$ is canonically trivialized.

   To see the  claim, recall that Corollary~\ref{steinberg general monodromy}  identifies the graded equivariant localization with respect to $\rho$ with coherent $\D$-modules  
   $$
\Coh_{[B(\alpha)\bs
G(\alpha)/B(\alpha)]}(\St_{\alpha, w\alpha}/G)^{\BS}_{loc} \simeq \D_{\coh}(B(\alpha)\bs
G(\alpha)/B(\alpha)).
$$
By definition, the $S^1$-action $\varphi_\alpha$ on the stack $B(\alpha)\bs
G(\alpha)/B(\alpha)$ is given by the central element  $\alpha$.  
By the functoriality of previous constructions, 
the induced action by $\cL\varphi_\alpha$ coincides with the induced action of $\varphi_\alpha$ on $\D$-modules.
Therefore it  can be canonically trivialized:
the
automorphism groups of a stack act homotopically on
$\D$-modules; and a fixed choice of logarithm $\lambda = \log(\alpha)$ determines a path $\exp(t\lambda)$ from the identity $e\in G(\alpha)$
to the central element $\alpha$. 

Finally, for a more symmetric statement, one could forget the grading of the above $\BS$-equivariance and only remember the underlying $S^1$-equivariance with respect to $\rho$. (Note that the $S^1$-action $\rho_\alpha$ does not
descend to $\Aff(S^1)$ or further extend to $\BS$ since it has a nontrivial semisimple part.) Then the above argument provides a (2-periodic)
equivalence
of the respective equivariant localizations.
   \end{remark}


\section{Langlands parameters as loop spaces}\label{Langlands section}


Motivated by an ongoing project to better understand representations
of real groups, we were led to a Galois-twisted version of the
relationship between loop spaces and Steinberg varieties developed in
the previous section.


\subsection{Involutions of flag varieties}

Let $\Gamma$ be a finite group.

A $\Gamma$-action on a
derived stack $Y$ is by definition a derived stack $\cY \to {\rm B}\Gamma$
equipped with an identification $\cY \times_{{\rm B}\Gamma} E\Gamma \simeq
Y$. We will usually write $Y/\Gamma$ in place of $\cY$. The
$\Gamma$-invariants (or $\Gamma$-fixed points) of $Y$ are by
definition the derived mapping stack of sections
$$ Y^\Gamma = \Map_{{\rm B}\Gamma}({\rm B}\Gamma, Y/\Gamma).
$$ (Note that the above discussion applies to $\Gamma$ an
arbitrary group derived stack.)

We will encounter the situation of two different $\Gamma$-actions on
single $Y$. To avoid confusion, we will consider two copies $\Gamma$
and $\Gamma'$ of the group indexed by the two actions, and write
$Y^\Gamma$ and $Y^{\Gamma'}$ to denote the invariants of the
respective actions.

Let us specialize to $\Gamma$ the cyclic group $\Z/2\Z$.

Consider the $\Gamma$-action on $G$ and in turn ${\rm B}G$ induced by an
involution $\eta$ of $G$. Note that the classifying space construction
and group invariants do not necessarily commute (the former is a
colimit and the latter is a limit). There is always a map $B(G^\Gamma)
\to ({\rm B}G)^\Gamma$, but in general $({\rm B}G)^\Gamma$ is a far richer object.

\begin{thm} \label{fixed points on BG}
There is a finite set $I$ of involutions of $G$ containing $\eta$ and a natural identification
$$
({\rm B}G)^\Gamma \simeq \coprod_{\iota\in I} BK_\iota
$$
where $K_\iota\subset G$ denotes the fixed-point subgroup of the involution $\iota$.
\end{thm}

\begin{proof}
Let $G\rtimes \Gamma$ be the semidirect product with respect to the given action.
Unwinding the definitions gives that $({\rm B}G)^\Gamma$ is equivalent to 
homomorphisms $\Gamma\to G\rtimes \Gamma$, up to conjugation, such that the composition
$\Gamma\to G\rtimes \Gamma \to \Gamma$ is the identity. 
Thus we have a natural identification
$$
({\rm B}G)^\Gamma \simeq \{ g\in G | g\eta(g) = 1\} /G
$$ We take $I$ to be the set of components of this stack -- the
set of conjugacy classes of solutions to the above equation -- and pick
one $\iota$ from each class.
\end{proof}

\begin{remark} 
We can view the quotient $\cG_{{\rm B}\Gamma} = G/\Gamma$ as a group scheme
over ${\rm B}\Gamma$ such that its base change along $E\Gamma \to {\rm B}\Gamma$
is identified with $G$. From this perspective, the invariants
$({\rm B}G)^\Gamma$ form the moduli of principal $\cG_{{\rm B}\Gamma}$-bundles
over ${\rm B}\Gamma$. Forgetting the invariant structure on such a bundle
corresponds to base change along $E\Gamma \to {\rm B}\Gamma$.
\end{remark}

Now consider the $\Gamma$-action on the fiber product
$$
\BGB \simeq {\rm B}B \times_{{\rm B}G} {\rm B}B\simeq G\backslash (G/B \times G/B)
$$ induced by the involution $\eta$ and the exchange of the two
factors.  As an immediate consequence of Theorem~\ref{fixed points on BG},
we can calculate the
$\Gamma$-fixed points in $\BGB$ in terms of quotients of $G/B$ by
symmetric subgroups.

\begin{corollary}\label{sym flags} There is a natural identification
$$ (\BGB)^\Gamma \simeq \coprod_{\iota\in I} K_\iota\backslash G/B.
$$
\end{corollary}

\begin{remark} 
Recall that we can view $\BGB \simeq {\rm B}B \times_{{\rm B}G} {\rm B}B$ as the moduli of $G$-local systems on the interval $[0,1]$ with $B$-reductions
at the end points $\{0,1\}$. 
From this perspective, 
 flipping $[0,1]$
around the midpoint $1/2 \in [0,1]$ induces
the involution that exchanges the two factors of the fiber product.

Furthermore, the involution $\eta$ descends the group $G$ to a group
scheme $ \cG_{[0, 1/2\rrbracket} = (G \times [0,1])/\Gamma$ over the
  quotient $[0,1/2\rrbracket = [0,1]/\Gamma$.  Finally, the invariants
    $(\BGB)^\Gamma$ form the moduli of $\cG_{[0, 1/2\rrbracket}$-local
      systems on $[0,1/2\rrbracket$ with $B$-reduction at the boundary
        point $\{0\}$.

As our notation suggests, we often regard $[0, 1/2\rrbracket$ as the
  interval $[0,1/2]$ with the end point $\{1/2\}$ equipped with the
  symmetries $\Gamma$. This makes apparent the identification of
  Corollary \ref{sym flags} via the intermediate realization
  $(\BGB)^\Gamma \simeq ({\rm B}G)^\Gamma \times_{{\rm B}G} {\rm B}B$.
\end{remark}


\subsection{Involutions of Steinberg varieties}

We continue to let $\Gamma$ denote the cyclic group $\Z/2\Z$. 


\begin{defn}
The {\em Langlands parameter variety} $\La^\eta$ is defined to be
$$
\La^\eta = \{ (g, B) \in G\times \cB | g\eta(g) \in B\}.
$$
\end{defn}

Consider the involution of the Grothendieck-Steinberg variety
$$
\St/G \simeq \cL({\rm B}B) \times_{\cL({\rm B}G)} \cL({\rm B}B)
$$
induced by the involution $\eta$ of ${\rm B}G$, the exchange of the two factors of $\cL({\rm B}B)$, and the antipodal map on $S^1$.
To reflect the fact that we are twisting by the antipodal map on $S^1$, we will write $\Gamma_{tw}$ for the resulting $\Z/2\Z$-action on $\St$.

\begin{thm} There is a natural identification
$$
\La^\eta/G \simeq (\St/G)^{\Gamma_{tw}}\simeq (\cL(\BGB))^{\Gamma_{tw}}.
$$
\end{thm}

\begin{proof}
This is a straightforward descent assertion from $G$-local systems on a cylinder with $B$-reductions along the two boundary circles to $\eta$-twisted $G$-local systems on the M\"obius strip with $B$-reduction
along the single boundary circle.
\end{proof}

\begin{remark}
Observe that $\cL/G$ is equipped with a canonical $S^1$-action such that the identification of the above theorem is $S^1$-equivariant.
\end{remark}

\begin{remark} 
Recall that we can view the equivariant Grothendieck-Steinberg  variety
$
\St/G\simeq \cL (B\bs G/B) 
$
as the moduli of $G$-local systems on the cylinder $Cyl = [0,1] \times S^1$ with $B$-reductions at 
the boundary circles $\{0,1\} \times S^1$. 
From this perspective, the above involution of the cylinder $Cyl$ has quotient
the M\"obius strip  $Moeb = ([0,1] \times S^1)/\Gamma_{tw}$.

Furthermore, the involution $\eta$ descends the group $G$ to a group scheme $\cG_{Moeb} = (G \times Cyl)/\Gamma_{tw}$ over  the M\"obius strip  $Moeb$.
We conclude that the invariants 
$
\La^\eta/G \simeq (\St/G)^{\Gamma_{tw}}
$ form the moduli of $\cG_{Moeb}$-local systems on $Moeb$ with $B$-reduction along the boundary circle $\{0\}\times S^1$.
\end{remark}


\subsection{Fixed monodromy}

We have the ordered eigenvalue map 
$$
\xymatrix{
\tilde \chi: \La^\eta \ar[r] & H
&
\tilde\chi (g, B) = [g\eta(g)]\in B/U \simeq H.
}
$$

For $\alpha\in H$, let $\wh H_\alpha$ denote its formal neighborhood.

\begin{defn}
The {\em monodromic Langlands parameter variety} $\La^\eta_\alpha$ is defined to be
$$
\La^\eta_\alpha =  \tilde\chi^{-1}(\wh H_\alpha). 
$$
\end{defn}

Note that $g\eta(g)$ and $\eta(g\eta(g))= \eta(g) g$ are conjugate by $g$. Thus $\alpha = [g\eta(g)]$
and $\eta(\alpha)$ are conjugate by $w\in W$. 

Consider the $\Gamma_{tw}$-action on 
$$
\St_{\alpha, \eta(\alpha)}/G\subset \St/G
$$
obtained by restriction.

\begin{lemma} \label{monod vs fixed} Fixing monodromy commutes with taking $\Gamma_{tw}$-invariants:
$$
\La^\eta_\alpha/G \simeq (\St_{\alpha, \eta(\alpha)}/G)^{\Gamma_{tw}}.
$$
\end{lemma}

\begin{proof}
Immediate from definitions.
\end{proof}

Now we arrive at the main result explaining the relation between Langlands parameters and loop spaces. It is an immediate consequence of our previous results and the following general observation.

Let $Y$ be a derived stack with $\Gamma$-action, and consider the
induced $\Gamma$ action on $\cL Y$. Let $\Gamma_{tw}$ denote the
$\Gamma$ action on $\cL Y$ obtained by twisting the induced action by
the antipodal map on $S^1$.

\begin{lemma}\label{key lemma} There are natural identifications
$$
(\cL^u Y)^\Gamma \simeq \cL^u (Y^\Gamma)
\qquad
(\cL^u Y)^\Gamma \simeq (\cL^u Y)^{\Gamma_{tw}}
$$
\end{lemma}

\begin{proof}
The first assertion is obvious from standard adjunctions.

For the second assertion, by standard adjunctions, we have equivalences
$$
(\cL^u Y)^\Gamma \simeq \Map_{\Gamma}(E\Gamma, \cL^u Y) \simeq 
\Map_{\Gamma}(E\Gamma, \Map({\rm B}\Ga, Y)) \simeq 
\Map_{\Gamma}({\rm B}\Ga, Y)
$$
and similarly for $\Gamma_{tw}$.
But the antipodal map on $S^1$ has degree $1$, hence induces the identity on $\cO(S^1) \simeq H^*(S^1, k)$.
Thus the induced map on the affinization $B\Ga \simeq \Aff(S^1) = \Spec \cO(S^1)$ is an equivalence.
\end{proof}

\begin{thm}\label{main result}
There is a natural identification
$$
\La^\eta_\alpha/G \simeq \cL^u((B(\alpha) \backslash G(\alpha) /B(\alpha))^\Gamma)
$$
\end{thm}

\begin{proof}
Follows immediately from  Theorem~\ref{monodromic steinberg variety as loop space}, Lemma~\ref{monod vs fixed}, and Lemma~\ref{key lemma}:
$$
\begin{array}{ccl}
\La^\eta_\alpha/G 
& \simeq & (\St_{\alpha, \eta(\alpha)}/G)^{\Gamma_{tw}}\\
& \simeq &  (\cL^u (B(\alpha)\bs G(\alpha)/B(\alpha)))^{\Gamma_{tw}}\\
& \simeq &  (\cL^u (B(\alpha)\bs G(\alpha)/B(\alpha)))^{\Gamma}\\
& \simeq & \cL^u ((B(\alpha)\bs G(\alpha)/B(\alpha))^\Gamma).
\end{array}
$$
\end{proof}

\begin{corollary}\label{final result} There is a natural identification
$$
\La^\eta_\alpha/G \simeq \coprod_{\iota\in I} \cL^u(K(\alpha)_ {\iota}\backslash G(\alpha)/B(\alpha)).
$$
\end{corollary}

\begin{proof}
Follows immediately from Corollary~\ref{sym flags}
and Theorem~\ref{main result}:
$$
\La^\eta_\alpha/G \simeq 
 \cL^u((B(\alpha) \backslash G(\alpha) /B(\alpha))^\Gamma)
\simeq \coprod_{\iota\in I} \cL^u(K(\alpha)_{\iota}\backslash G(\alpha)/B(\alpha)).
$$
\end{proof}

\begin{remark} 
For simplicity, let us focus on the assertion of Theorem~\ref{main result} when $\alpha$ is the identity $e\in H$.
Outside of relatively simple deductions, the argument boils down to the identification
$$
(\cL^u(B \backslash G /B))^\Gamma \simeq (\cL^u(B \backslash G /B))^{\Gamma_{tw}}.
$$
We can view this as an equivalence between twisted {\em unipotent} local systems on the product 
$[0,1/2\rrbracket \times S^1 = ([0,1]/\Gamma) \times S^1$ and on the M\"obius strip $Moeb = ([0,1]\times S^1)/\Gamma_{tw}$. Such an equivalence would not be true for all local systems, but 
Lemma~\ref{key lemma} explains that the twist by the antipodal map of $S^1$ is innocuous when it comes to unipotent local systems.

Finally, we can view Corollary~\ref{final result} as a natural consequence of this interplay between
 the product 
$[0,1/2\rrbracket \times S^1$ and the M\"obius strip $Moeb$. Namely, we have seen that the left hand side can be naturally interpreted as twisted unipotent local systems on $Moeb$, while the right hand side can be naturally interpreted as twisted unipotent local systems on $[0,1/2\rrbracket \times S^1$.
 \end{remark}

\subsection{Application to Langlands parameters}


Let us focus on the most interesting case of trivial monodromy when $\alpha$ is the identity $e\in H$. 
By Corollary~\ref{final result},
 we have a canonical identification
 $$
\La^\eta_e/G \simeq \coprod_{\iota\in I} \cL^u(K_ {\iota}\backslash G/B).
$$
  where the monodromic Langlands parameter variety $\La^\eta_{e}$ consists of the data
  $$
\La^\eta_{e} = \{(g, B) \in G\times \cB\, |\, g\eta(g) \in G^u \cap B\},
$$
where $G^u\subset G$ is the formal neighborhood of the unipotent elements.
 
Now our results from \cite{conns}, as recalled in 
Section~\ref{section: background}, immediately provide
    a canonical equivalence of dg categories
$$
\Coh_{[X_e/G]}(\La^\eta_e/G)^{\BS}_{loc} \simeq 
\oplus_{\iota\in I} \D_\coh(K_ {\iota}\backslash G/B).
$$
where we set $X_e/G =\coprod_{\iota\in I} K_ {\iota}\backslash G/B$.
Observe that the dg dervied category  
$\oplus_{\iota\in I} \D_\coh(K_ {\iota}\backslash G/B)$
forms the categorical Langlands parameters for trivial infinitesimal character.

Thus we arrive at the following fundamental interpretation of
categorical Langlands parameters.

\begin{corollary}\label{trivmonresult}
The localized $\BS$-invariants  in the dg category $\Coh_{[X_e/G]}(\La^\eta_e/G)$ are canonically equivalent
to the categorical Langlands parameters for trivial infinitesimal character.
\end{corollary}

For general monodromy $\alpha\in H$,
again by Corollary~\ref{final result},
we have a canonical identification
  $$
\La^\eta_\alpha/G 
\simeq \coprod_{\iota\in I} \cL^u(K(\alpha)_{\iota}\backslash G(\alpha)/B(\alpha)).
$$
Thus we can transport the canonical $\BS$-action on the unipotent loop space of the right hand side to the monodromic Langlands parameter variety of the left hand side.  Again our results from \cite{conns} immediately provide
    a canonical equivalence of $\oo$-categories
$$
\Coh_{[X_\alpha/G]}(\La^\eta_\alpha/G )^{\BS}_{loc}
\simeq \oplus_{\iota\in I} \D_\coh(K(\alpha)_{\iota}\backslash G(\alpha)/B(\alpha))
$$
where  we set $X_\alpha/G =\coprod_{\iota\in I} K(\alpha)_{\iota}\backslash G(\alpha)/B(\alpha)$.
Thus we arrive at the following interpretation of Langlands parameters for arbitrary regular infinitesimal character.

\begin{corollary}\label{genmonresult}
  For any regular $\alpha\in H$,
  the localized $\BS$-invariants in the dg category
  $
\Coh_{[X_\alpha/G]}(\La^\eta_{\alpha}/G)
$
are canonically equivalent to the categorical Langlands parameters at infinitesimal character $W\cdot\alpha$.
\end{corollary}

The above result gives a description of $\D$-modules on the
geometric parameter spaces 
$$\coprod_{\iota\in I} K(\alpha)_\iota\backslash G(\alpha)/B(\alpha)
$$
as part of a nice family with respect to the paramater $\alpha$.
Namely, they can be recovered from quasicoherent sheaves on the unipotent loop spaces of
 the
geometric parameter spaces. And the unipotent loop spaces in turn
fit into the nice family formed by the Langlands
parameter variety $\La^\eta$. It is crucial
that we sought such a uniform picture in the realm of loop spaces
rather than cotangent bundles.

\begin{remark}
By 
Lemma~\ref{monod vs fixed}, we have a canonical identification 
$$
\La^\eta_\alpha/G \simeq (\St_{\alpha, \eta(\alpha)}/G)^{\Gamma_{tw}}.
$$
equipping $\cL_\alpha/G$ with a canonical $S^1$-action.
As we discussed in Remark~\ref{comparison of actions} in the setting of Steinberg varieties, one should proceed carefully with respect to this symmetry.
In general, this $S^1$-action on 
 $
 \La^\eta_{\alpha}/G
 $ and the $B\Ga$-action induced from its realization as a unipotent loop space do not exactly coincide.
Rather the action map of one can be obtained from that of the other by the universal automorphism
given by the central loop $\alpha$.
\end{remark}




\begin{thebibliography}{BBDG}







\bibitem[ABV]{ABV}
J. Adams, D. Barbasch, and D. Vogan, The Langlands Classification
and Irreducible Characters for Real Reductive Groups. {\em Progress
in Mathematics} {\bf 104}. Birkhauser, Boston-Basel-Berlin, 1992.



\bibitem[BGS]{BGS} A. Beilinson, V. Ginzburg and W. Soergel, Koszul
  duality patterns in representation theory. {\em Jour. Amer. Math.
    Soc.}  {\bf 9} (1996), no. 2, 473--527.


\bibitem[BFN]{BFN} D. Ben-Zvi, J. Francis and D. Nadler, Integral transforms
and Drinfeld centers in derived algebraic geometry. Preprint arXiv:0805.0157.
{\em Jour. Amer. Math. Soc.} {\bf 23} (2010), 909--966.

\bibitem[BN07]{BN07} D. Ben-Zvi and D. Nadler, Loop Spaces and Langlands Parameters. 
Preprint arXiv:0706.0322.

\bibitem[BN1]{character} D. Ben-Zvi and D. Nadler, The character
  theory of a complex group. Preprint arXiv:0904.1247.

\bibitem[BN2]{conns} D. Ben-Zvi and D. Nadler, Loop Spaces and Connections. 
arXiv:1002.3636. J. Topology (2012) 5(2) 377-430.




\bibitem[Be]{Roma ICM} R. Bezrukavnikov, Noncommutative
counterparts of the Springer resolution. arXiv:math.RT/0604445.
 International Congress of Mathematicians. Vol. II,  1119--1144, Eur. Math. Soc., Z\"urich, 2006.


\bibitem[BY]{BY} R. Bezrukavnikov and Z. Yun, On Koszul duality for Kac-Moody groups.  e-print arXiv:1101.1253.

\bibitem[CG]{CG} N. Chriss and V. Ginzburg,
Representation theory and complex geometry. Birkh\"auser Boston,
Inc., Boston, MA, 1997. x+495 pp.

\bibitem[G]{Ginzburg} V. Ginzburg, Isospectral commuting variety and
  the Harish-Chandra $\D$-module. Preprint arXiv:1002.3311.

\bibitem[GKM]{GKM} M. Goresky, R. Kottwitz and R. MacPherson, Equivariant
cohomology, Koszul duality, and the localization theorem.  Invent.
Math.  131  (1998),  no. 1, 25--83.



\bibitem[KL]{KL} D. Kazhdan and G. Lusztig, Proof of the Deligne-Langlands
conjecture for Hecke algebras. {\em Invent. Math.}  {\bf 87} (1987),
no. 1, 153--215.








\bibitem[SV1]{SV1} O. Schiffmann and E. Vasserot, The elliptic Hall
  algebra, Cherednick Hecke algebras and Macdonald polynomials. arXiv:0802.4001.
Compos. Math. 147 (2011), no. 1, 188-234.

\bibitem[SV2]{SV2} O. Schiffmann and E. Vasserot, The elliptic Hall
  algebra and the equivariant K-theory of the Hilbert scheme of
  $\mathbb{A}^2$. Preprint arXiv:0905.2555.



\bibitem[So]{Soergel} W. Soergel, Langlands' philosophy and Koszul duality,
      Algebra--representation theory (Constanta, 2000), {\em NATO Sci. Ser.
      II Math. Phys. Chem.}, vol. {\bf 28}, Kluwer Acad. Publ., Dordrecht,
      2001, pp. 379-414.

\bibitem[To]{Toen} B. To\"en, Higher and Derived Stacks: a global overview.
e-print arXiv:math.AG/0604504.
 Algebraic geometry---Seattle 2005. Part 1,  435--487, 
 Proc. Sympos. Pure Math., 80, Part 1, Amer. Math. Soc., 
 2009.








\bibitem[V]{Vogan} D. Vogan,
Irreducible characters of semisimple Lie groups. IV.
Character-multiplicity duality.  {\em Duke Math. J.}  {\bf
  49}  (1982), no. 4, 943--1073.



 \end{thebibliography}
\end{document}